\documentclass[a4paper, 11pt]{article}

\usepackage{amsfonts,enumerate}
\usepackage{amsmath}
\usepackage{enumerate}
\usepackage{graphicx}
 \usepackage[colorlinks=true]{hyperref}

\newtheorem{proposition}{Proposition}[section]

\newenvironment{proof}{\smallskip\noindent\emph{Proof.}\hspace{1pt}}%
{\hspace{-5pt}{\nobreak\quad\nobreak\hfill\nobreak$\square$\vspace{8pt}%
\par}\smallskip\goodbreak}

\newcommand{\Section}[1]{\section{#1}\setcounter{equation}{0}}

\newcommand{\C}[1]{\mathbf{C^{#1}}}

\newcommand{\reali}{{\mathbb{R}}}
\newcommand{\interi}{{\mathbb{Z}}}
\newcommand{\naturali}{{\mathbb{N}}}

\renewcommand{\epsilon}{\varepsilon}
\renewcommand{\phi}{\varphi}
\renewcommand{\theta}{\vartheta}

\newcommand{\vmax}{V_{\max}}

\begin{document}

\title{The Godunov Method for a $2$--Phase Model}

\author{M. Garavello\thanks{Dipartimento di Matematica e Applicazioni,
    Universit\`a di Milano Bicocca, Via R. Cozzi 55, 20125 Milano, Italy.
    E-mail: \texttt{mauro.garavello@unimib.it}} \and
  F. Marcellini\thanks{Dipartimento di Matematica e Applicazioni,
    Universit\`a di Milano Bicocca, Via R. Cozzi 55, 20125 Milano, Italy.
    E-mail: \texttt{francesca.marcellini@unimib.it}}}

\maketitle

\begin{abstract}
  We consider the Godunov numerical method to the phase-transition 
  traffic model, proposed in~\cite{ColomboMarcelliniRascle},
  by Colombo, Marcellini, and Rascle. Numerical tests are shown
  to prove the validity of the method.
  Moreover we highlight the differences between
  such model and the one proposed in~\cite{BlandinWorkGoatinPiccoliBayen},
  by Blandin, Work, Goatin, Piccoli, and Bayen. 

  \noindent\textit{2000~Mathematics Subject Classification:} 35L65,
  90B20

  \medskip


\end{abstract}

\Section{Introduction}
\label{sec:Intro}

Aim of this paper is to present the Godunov method for approximating
the solutions to the $2$--phase traffic model, 
introduced in~\cite{ColomboMarcelliniRascle} by Colombo, Marcellini, and
Rascle. The model consists on the system of conservation laws
\begin{equation}
  \label{eq:Modeleta}
  \left\{
    \begin{array}{l}
      \partial_t \rho +
      \partial_x \left( \rho\, v (\rho,\eta) \right) = 0
      \vspace{.2cm}\\
      \partial_t \eta +
      \partial_x \left( \eta\, v (\rho, \eta) \right) = 0
    \end{array}
  \right.
  \quad \mbox{with} \quad
  v(\rho, \eta)
  =
  \min \left\{ \vmax, \frac{\eta}{\rho}\, \psi(\rho) \right\},
\end{equation}
where $\rho = \rho(t,x)$ is the car density at time $t>0$
and at position $x \in \reali$, 
$\eta = \eta(t,x)$ is a generalized
momentum, $v = v(\rho, \eta)$ is the avarage speed of cars,
$\vmax$ is a positive constant describing the maximum speed of cars,
and $\psi = \psi(\rho)$ is a given decreasing function.
This model is an extension of the famous
Lighthill-Whitham-Richards (LWR) model, see~\cite{LighthillWhitham, Richards},
and it is obtained by assuming that different kinds of drivers
have different maximal speeds $w = w\left(\rho, \eta\right)$, where
$w=\eta / \rho \in \left[ \check{w}, \hat{w} \right]$,
for suitable constants $0 < \check{w} < \hat{w}$.

The model described by~(\ref{eq:Modeleta}) belongs to the class of hyperbolic
phase transition models for traffic, whose aim is to describe
the different traffic regimes; see~\cite{kerner2012physics}.
The pioneer work in this class is the model proposed in 2002
by Colombo; see~\cite{Colombo1.5, MR2032809, ColomboGoatinPriuli}.
Its distinctive feature is that the two phases are disconnected. 
Subsequently other phase transition models have been introduced
in the literature;
see~\cite{BlandinWorkGoatinPiccoliBayen, ColomboMarcelliniRascle,
  Goatin2Phases, LebacqueMammarHajSalem, Marcellini}.
In particular, the models
in~\cite{BlandinWorkGoatinPiccoliBayen, ColomboMarcelliniRascle},
although very similar in the fundamental diagram, are indeed different,
since solutions to Riemann problems have different structures and
contain a different number of waves.
Also the derivation of the two models is completely different:
the construction in~\cite{BlandinWorkGoatinPiccoliBayen} is done imposing
a priori two phases, the free and the congested one,
while in~\cite{ColomboMarcelliniRascle} the two phases are obtained
as a consequence of the speed limit $\vmax$.

In the present paper, we describe the Godunov numerical scheme
for system~(\ref{eq:Modeleta}).
The Godunov method is a finite volume method for conservation laws;
it is based on the solution to the Riemann problem in order to
approximate the flux between two contiguous cells;
see~\cite{MR0119433, MR1925043}.
In our setting, the fundamental diagram in the conserved variable is
convex; hence the application of this method is straightforward.
Instead the nonconvexity of the fundamental diagram for the model
in~\cite{BlandinWorkGoatinPiccoliBayen} is a source of various problems
for the Godunov method; see~\cite{MR2453129, MR1378558}.

The paper is organized as follow.
In Section~\ref{sec:Des} we describe the 2-Phase
Traffic Model~(\ref{eq:Modeleta}). 
In Section~\ref{sec:Comparison} we make a comparison between
the $2$--Phase model~(\ref{eq:Modeleta}) and that
in~\cite{BlandinWorkGoatinPiccoliBayen}.
Finally in Section~\ref{sec:Godunov}, we describe the Godunov method
to the model~(\ref{eq:Modeleta}).
Some numerical integrations conclude the paper.

\Section{Notations and Description of the Phase Transition Model}
\label{sec:Des}

In this section we fix notations and we recall some properties concerning the $2$--Phase traffic model introduced in~\cite{ColomboMarcelliniRascle}. As already said, the model~(\ref{eq:Modeleta}) is an extension of the classical LWR model, given by the following scalar conservation law
\begin{equation*}
  \partial_t \rho + \partial_x \left( \rho \, V \right) =0 \,,
\end{equation*}
where $\rho$ is the traffic density and $V=V(t,x,\rho)$ is the speed. 

We consider the folllowing two assumptions on the speed:
\begin{itemize}
\item We assume that, at a given density, different drivers may choose different velocities, that is, we assume that $V = w \, \psi(\rho)$, where $\psi = \psi(\rho)$ is a $\C{2}$ function and $w = w(t,x)$ is the maximal speed of a driver, located at position $x$ at time $t$.
\item We impose an overall bound on the speed $\vmax$.
\end{itemize}
We get the following $2 \times 2$ system
\begin{equation}
  \label{eq:rhow}
  \left\{
    \begin{array}{l}
      \partial_t \rho + \partial_x ( \rho v ) =0
      \vspace{.2cm}\\
      \partial_t w + v \, \partial_x w =0
    \end{array}
  \right.
  \qquad \mbox{ with } \qquad
  v = \min \left\{\vmax ,\, w \, \psi(\rho) \right\}\,,
\end{equation}
and then the model in~(\ref{eq:Modeleta}), with the change of variables $\eta = \rho w$.

As a consequence of the introduction of the speed bound $\vmax$, system~(\ref{eq:Modeleta}) produces two phases, the free and the congested one, described by the sets
\begin{eqnarray}
  \label{eq:phF}
  F
  & = &
  \left\{
    (\rho,w) \in [0,R] \times [\check w, \hat w]
    \colon v(\rho, \rho w) = \vmax
  \right\},
  \\
  \label{eq:phC}
  C
  & = &
  \left\{
    (\rho,w) \in [0,R] \times [\check w, \hat w]
    \colon v(\rho, \rho w) = w \, \psi(\rho)
  \right\}\,.
\end{eqnarray}
In Figure~\ref{fig:phases} it is represented the fundalmental digram in the coordinates $(\rho,\eta)$ and $(\rho, \rho v)$. Note that $F$ and $C$ are closed sets and $F \cap C \neq \emptyset$. Note also that $F$ is one-dimensional in the $(\rho,  \rho v)$ plane of the fundamental diagram, while it is two-dimensional in the $(\rho,\eta)$ coordinates. 
\begin{figure}[t!]
\centering
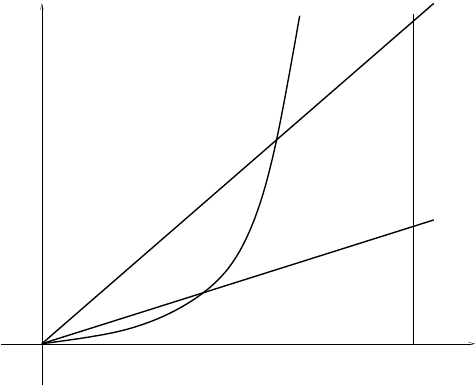
\hfil
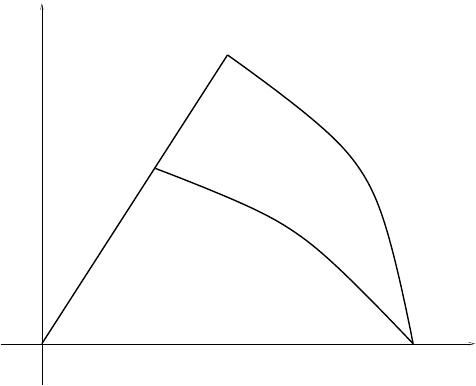
\caption{The free phase $F$ and the congested phase $C$ resulting
from\/ {\rm(\ref{eq:Modeleta})} in the coordinates, from left to right, $(\rho,\eta)$ and $(\rho, \rho v)$.}\label{fig:phases}
\end{figure}  

We recall the following assumptions:
\begin{description}

\item[(H-1)] $R,\check w, \hat w, \vmax$ are positive
  constants, with $\vmax < \check w < \hat w$.

\item[(H-2)]$\psi \in \C{2} \left( [0,R];[0,1]\right)$ is such
that\[
    \begin{array}{@{}rcl@{\qquad}rcl@{}}
      \psi(0) \!\!&\!\! = \!\!&\!\! 1,
      \!\!&\!\!
      \psi(R) \!\!&\!\! =\!\!&\!\! 0,
      \\
      \psi'(\rho) \!\!&\!\! \leq \!\!&\!\! 0,
      \!\!&\!\!
      \displaystyle
      \frac{d^2\ }{d\rho^2} \left( \rho\, \psi(\rho) \right) \!\!&\!\! \leq \!\!&\!\!
      0
      \quad \mbox{ for all } \rho \in [0, R]\,.
    \end{array}
\]
\item[(H-3)] Waves of the first family in the congested phase $C$ have negative speed.
\end{description}

We recall the eigenvalues, right eigenvectors, and Lax curves $\eta=\mathcal{L}_i(\rho;\rho_o,\eta_o)$ 
in~$C$:\[
    \begin{array}{@{}rcl@{\quad}rcl@{}}
      \lambda_{1} (\rho, \eta)
      \!\!&\!\! = \!\!&\!\!
      \eta\, \psi'(\rho) + v(\rho, \eta),
      &
      \lambda_{2} (\rho, \eta)
      \!\!&\!\! = \!\!&\!\!
      v(\rho, \eta),
      \\[5pt]
      r_{1} (\rho, \eta)
      \!\!&\!\! = \!\!&\!\!
      \left[
        \begin{array}{c}
          -\rho
          \\
          -\eta
        \end{array}
      \right],
      &
      r_{2} (\rho, \eta)
      \!\!&\!\! = \!\!&\!\!
      \left[
        \begin{array}{c}
          1
          \\
          \eta\left( \frac{1}{\rho}-\frac{\psi'(\rho) }{\psi(\rho) }\right)
        \end{array}
      \right],
      \\
      \nabla \lambda_1 \cdot r_1
      \!\!&\!\! = \!\!&\!\!
      \displaystyle
      -\frac{d^2\ }{d\rho^2} \left[ \rho\, \psi(\rho) \right],
      &
      \nabla \lambda_2 \cdot r_2
      \!\!&\!\! = \!\!&\!\!
      0,
      \\
      \mathcal{L}_1(\rho;\rho_o,\eta_o)
      \!\!&\!\! = \!\!&\!\!
      \displaystyle
      \eta_o \frac{\rho}{\rho_o},
      &
      \mathcal{L}_2(\rho;\rho_o,\eta_o)
      \!\!&\!\! = \!\!&\!\!
      \displaystyle
      \frac{\rho \, v(\rho_o, \eta_o)}{\psi(\rho)},
      \; \rho_o < R. 
\end{array}
\]
When $\rho_o = R$, the 2-Lax curve through $(\rho_o, \eta_o)$ is
  the segment $\rho=R$, $\eta \in [R \check w, R \hat w]$.
  
In order to apply th Godunov's method in Section~\ref{sec:Godunov}, we recall the description of the solutions of the Riemann problem for the model~(\ref{eq:Modeleta}). First, we recall all the possible waves in the solution.
\begin{itemize}
\item \textsl{A Linear wave} is a wave connecting two states in the free
  phase.

\item \textsl{A Phase Transition Wave} is a wave connecting a left state
  $\left(\rho_l, \eta_l\right) \in F$ with a right state
  $\left(\rho_r, \eta_r\right) \in C$ satisfying 
  $\frac{\eta_l}{\rho_l} = \frac{\eta_r}{\rho_r}$.

\item \textsl{A Wave of the First Family} is a wave connecting a left state 
  $\left(\rho_l, \eta_l\right) \in C$ with a right state
  $\left(\rho_r, \eta_r\right) \in C$ such that
  $\frac{\eta_l}{\rho_l} = \frac{\eta_r}{\rho_r}$.

\item \textsl{A Wave of the Second Family} is a wave connecting a left state 
  $\left(\rho_l, \eta_l\right) \in C$ with a right state
  $\left(\rho_r, \eta_r\right) \in C$ such that
  $v\left(\rho_l, \eta_l\right) = v\left(\rho_r, \eta_r\right)$.
\end{itemize}    
  
\subsection{The Riemann Problem}
\label{sse:RP}
Under the assumptions \textbf{(H-1)}, \textbf{(H-2)} and \textbf{(H-3)}, for all states $(\rho^l,\eta^l)$,
  $(\rho^r, \eta^r) \in F \cup C$, the Riemann problem consisting
  of~(\ref{eq:Modeleta}) with initial data
  \begin{equation}
    \label{eq:RD}
    \rho(0,x) = \left\{
      \begin{array}{l@{\quad\mbox{ if }}rcl}
        \rho^l & x & < & 0
        \\
        \rho^r & x & > & 0
      \end{array}
    \right.
    \qquad
    \eta(0,x) = \left\{
      \begin{array}{l@{\quad\mbox{ if }}rcl}
        \eta^l & x & < & 0
        \\
        \eta^r & x & > & 0
      \end{array}
    \right.
  \end{equation}
  admits a unique self similar weak solution $(\rho,\eta) =
  (\rho,\eta) (t,x)$ constructed as follows:
  \begin{enumerate}[(1)]
  \item If $(\rho^l,\eta^l), (\rho^r,\eta^r) \in F$, then the solution consists of a linear wave separating $(\rho^l,\eta^l)$ from $(\rho^r,\eta^r)$.  
  \item If $(\rho^l,\eta^l), (\rho^r,\eta^r) \in C$, then the solution consists of a wave of the first family (shock or rarefaction) between $(\rho^l, \eta^l)$ and a middle state $(\rho^m, \eta^m)$, followed by a wave of the second family between $(\rho^m, \eta^m)$ and $(\rho^r, \eta^r)$.

  \item If $(\rho^l,\eta^l) \in C$ and $(\rho^r,\eta^r) \in F$,
    then the solution consists of a wave of the first family separating
    $(\rho^l, \eta^l)$ from a middle state $(\rho_\circ, \eta_\circ)$
    (which belongs to the intersection between $F$ and $C$)
    and by a linear wave separating $(\rho_\circ, \eta_\circ)$
    from $(\rho^r,\eta^r)$.
  \item If $(\rho^l,\eta^l) \in F$ and $(\rho^r,\eta^r) \in C$, then the solution consists of a phase transition wave between $(\rho^l, \eta^l)$ and a middle state $(\rho^m, \eta^m)$ (which is in $C$), followed by a wave of the second family between $(\rho^m, \eta^m)$ and $(\rho^r, \eta^r)$. 
\end{enumerate}
      
%
%
%
%
\Section{Comparison with the Model in~\cite{BlandinWorkGoatinPiccoliBayen}}
\label{sec:Comparison}

In this section we briefly describe the $2$--Phase traffic model,
introduced by Blandin, Work, Goatin, Piccoli and Bayen
in~\cite{BlandinWorkGoatinPiccoliBayen}, and we show the main differences
with~(\ref{eq:Modeleta}).

The phase transition traffic model in~\cite{BlandinWorkGoatinPiccoliBayen}
has been developed as an extension of the Colombo phase transition
model~\cite{Colombo1.5} and can be described by the following system
\begin{equation}
  \label{eq:BWGPB}
  \begin{array}{ll}
    \left\{
      \begin{array}{ll}
         \partial_t \rho +  \partial_x (\rho V) = 0, &
        \textrm{ if } (\rho,q) \in \Omega_f,\vspace{.2cm} \\
        \left\{ 
          \begin{array}{l}
             \partial_t \rho +  \partial_x (\rho v(\rho,q)) = 0,\\
             \partial_t q +  \partial_x (q v(\rho,q)) = 0.
          \end{array}
        \right.
        & \textrm{ if } (\rho,q) \in \Omega_c,
    \end{array}   
    \right.
  \end{array}
\end{equation}
where $\rho$ and $q$, represent respectively the car traffic density
and the linearized momentum,
$v = v\left(\rho, q\right)$ is the average speed of cars,
while the sets $\Omega_f$ and $\Omega_c$,
respectively the free and the congested
phases, are given by
\begin{equation}
  \label{eq:phases_pg}
    \begin{array}{l}
      \Omega_f = \left\{
        (\rho, q) \in [0,R] \times [0, +\infty[ :
        \, q = \frac{R (\rho - \sigma)}{\sigma (R - \rho)},\,\,
        0 \le \rho \le \sigma_+
      \right\}\vspace{.2cm} \\
      \Omega_c = \left\{
        (\rho, q) \in [0,R] \times [0, +\infty[ :
        \, v(\rho,q) \le V, \,\,\frac{q^-}R \le \frac{q}{\rho} \le \frac{q^+}R
      \right\}\,.
    \end{array}
\end{equation}
In~(\ref{eq:phases_pg}), $R > 0$ denotes the maximal density,
$V > 0$ is the maximal velocity attained by the vehicles
in the free phase, $q^-$ and $q^+$ are respectively the minimum and
the maximum value of the momentum $q$, and the constants $\sigma$ and
$\sigma_+$ are suitable constants in $[0,R]$.
The velocity $v: \Omega_f \cup \Omega_c  \to [0, +\infty[$ is defined by
\begin{equation}
  \label{eq:speed_pg}
  v (\rho, q) = \left\{ 
    \begin{array}{l@{\quad}l}
      V, & \textrm{ if } (\rho, q) \in \Omega_f
      \vspace{.2cm}\\
      v^{eq} (\rho) (1+ q)=
      \frac{V \sigma}{R - \sigma} \left(\frac{R}{\rho} -1
      \right) \left( 1+q\right), & \textrm{ if } (\rho, q) \in \Omega_c\,,
    \end{array}
  \right.
\end{equation}
where the \emph{equilibrium velocity} $v^{eq}$ represents
the desired speed of cars when the density is $\rho$. In the congested
phase $\Omega_c$, when $q = 0$, the velocity of cars is 
the equilibrium velocity $v^{eq}$.

The models~(\ref{eq:Modeleta}) and~(\ref{eq:BWGPB}) have some similarities.
Both models are described by two intersecting phases: the free phase
and the two-dimensional congested phase.
In Figures~\ref{fig:phases} and~\ref{fig:comparison_FD} the fundamental
diagrams, respectively for~(\ref{eq:Modeleta}) and for~(\ref{eq:BWGPB}),
are drawn.
Note however that the free phase in~(\ref{eq:BWGPB}) is one-dimensional 
both in the conserved quantity coordinates $\left(\rho, q\right)$
and in the coordinates $\left(\rho, \rho v\right)$.
In the model~(\ref{eq:Modeleta}) the free phase is two-dimensional
in the conserved quantity coordinates $\left(\rho, \eta\right)$.
\begin{figure}[t!]
  \centering
  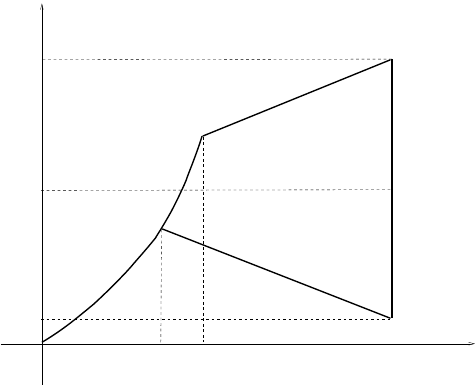
  \hfil
  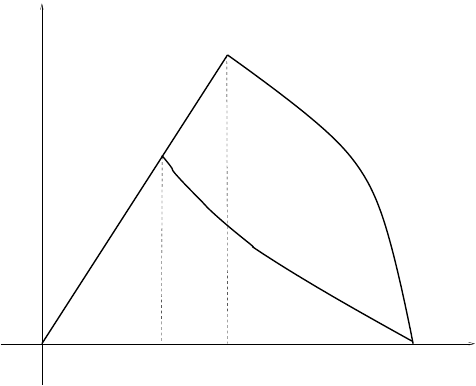
  \caption{The fundamental diagram in~\cite{BlandinWorkGoatinPiccoliBayen}
    in the coordinates, from left to right, $(\rho,q)$ and $(\rho, \rho v)$.
    Note that free phase $\Omega_f$ is one-dimensional, while the congested
    phase $\Omega_c$ is two-dimensional.}
  \label{fig:comparison_FD}
\end{figure} 

The main difference between~(\ref{eq:Modeleta}) and~(\ref{eq:BWGPB})
lies in the solution of the Riemann problem, in particular
when the left state belongs to the free phase and
the right state to the congested one.

The following result for~(\ref{eq:Modeleta}) holds.
\begin{proposition}
  \label{prop:1}
  Consider the system~(\ref{eq:Modeleta}) and fix
  the states $(\rho^l,\eta^l) \in F$, $(\rho^r,\eta^r) \in C$.
  The Riemann problem for~(\ref{eq:Modeleta}) with
  the initial condition
  \begin{equation*}
    \left(\rho_0, \eta_0\right)(x) =
    \left\{
      \begin{array}{ll}
        (\rho^l,\eta^l) & \textrm{ if } x < 0
        \\
        (\rho^r,\eta^r) & \textrm{ if } x > 0
      \end{array}
    \right.
  \end{equation*}
  is solved with at most two waves.
\end{proposition}
\begin{proof}
  Since $(\rho^l,\eta^l) \in F$ and $(\rho^r,\eta^r) \in C$, the solution
  to the Riemann problem consists of a shock wave
  connecting $(\rho^l, \eta^l)$ to a middle state $(\rho^m, \eta^m)$, 
  possibly followed by a wave of the second family
  connecting $(\rho^m, \eta^m)$ to $(\rho^r, \eta^r)$; see
  Subsection~\ref{sse:RP} and Figure~\ref{fig:comparison_RP}.
  This concludes the proof.
\end{proof}

The following result for~(\ref{eq:BWGPB}) holds.
\begin{proposition}
  \label{prop:2}
  Consider the system~(\ref{eq:BWGPB}) and fix
  the states $(\rho^l,q^l) \in \Omega_f$, $(\rho^r,q^r) \in \Omega_c$.
  The Riemann problem for~(\ref{eq:BWGPB}) with
  the initial condition
  \begin{equation*}
    \left(\rho_0, q_0\right)(x) =
    \left\{
      \begin{array}{ll}
        (\rho^l,q^l) & \textrm{ if } x < 0
        \\
        (\rho^r,q^r) & \textrm{ if } x > 0
      \end{array}
    \right.
  \end{equation*}
  is solved with at most three waves.
\end{proposition}

\begin{proof}
  We use here the notations in~\cite{BlandinWorkGoatinPiccoliBayen}.
  First define the intermediate state
  $(\rho^m,q^m) \in \Omega_{c}$ as he solution to the system
  \begin{equation*}
    \left\{
      \begin{array}{l}
        \frac{q^{m}}{\rho^{m}}=\frac{q}{R}
        \\
        v_{c}(\rho^m,q^m)= v_{c}(\rho^r,q^r)\,.
      \end{array}
    \right.
  \end{equation*}
  Let $\Lambda\left((\rho^l,q^l),(\rho^m,q^m) \right)$ be
  the speed of the phase transition wave connecting $(\rho^l,q^l)$ to
  $(\rho^m,q^m)$. Its value is given by the Rankine-Hugoniot condition
  \begin{equation*}
    \Lambda\left((\rho^l,q^l),(\rho^m,q^m) \right) =
    \frac{\rho^m v\left(\rho^m, q^m\right) - \rho^l V}{\rho^m - \rho^l}.
  \end{equation*}
  Denoting with $\lambda_{1}(\rho^m,q^m)$ the first eigenvalue of the
  Jacobian matrix of the flux at $(\rho^m,q^m)$,
  the following possibilities hold.
  \begin{enumerate}
  \item $\Lambda\left((\rho^l,q^l),(\rho^m,q^m) \right)\geq
    \lambda_{1}(\rho^m,q^m)$.
    In this case, the solution to the Riemann problem
    consists of a phase transition wave connecting
    $(\rho^l,q^l)$ to $(\rho^m,q^m)$, possibly followed by a 
    contact discontinuity wave connecting 
    $(\rho^m,q^m)$ to $(\rho^r,q^r)$; see Figure~\ref{fig:comparison_RP}.
    
  \item $\Lambda\left((\rho^l,q^l),(\rho^m,q^m) \right) <
    \lambda_{1}(\rho^m,q^m)$. Define
    $(\rho^p,q^p) \in \Omega_c$ the solution to
    \begin{equation*}
      \left\{
        \begin{array}{l}
          \frac{q^{p}}{\rho^{p}}=\frac{q^-}{R}
          \\
          \Lambda\left((\rho^l,q^l),(\rho^p,q^p) \right) 
          = \lambda_{1}(\rho^p,q^p)\,.
        \end{array}
      \right.
    \end{equation*}
    The solution to the Riemann problem 
    consists of a phase transition wave connecting $(\rho^l,q^l)$ to
    $(\rho^p,q^p)$, of a rarefaction wave of the first
    family connecting $(\rho^p,q^p)$ to
    $(\rho^m,q^m)$, and of contact discontinuity wave connecting
    $(\rho^m,q^m)$ to $(\rho^r,q^r)$; see Figure~\ref{fig:comparison_RP}.
  \end{enumerate}
  The proof is so concluded.
  \begin{figure}[t!]
    \centering
    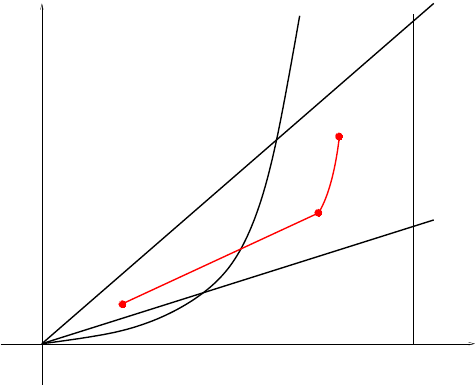
    \hfil
    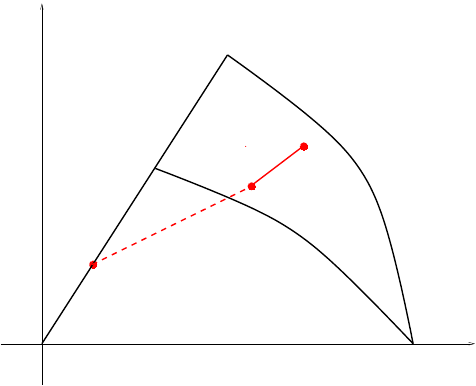
    \hfil
    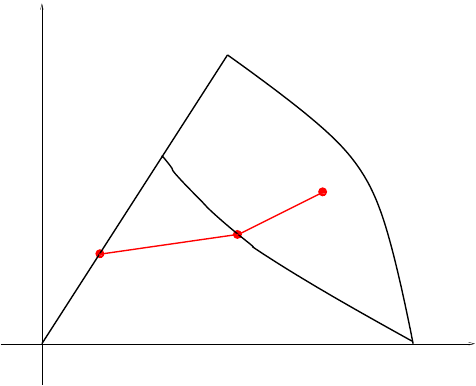
    \hfil
    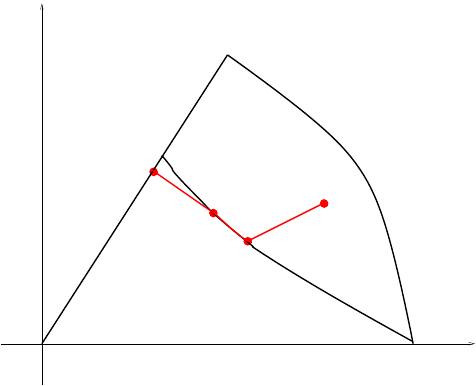
    \caption{Above, the Riemann problem for~(\ref{eq:Modeleta})
      if $(\rho^l,\eta^l) \in F$ and $(\rho^r,\eta^r) \in C$, left, in the 
      $(\rho,\eta)$ coordinates, right, in the $(\rho,\rho v)$ coordinates.
      Below, the Riemann problem for~(\ref{eq:BWGPB})
      if  $(\rho^l,q^l)\in \Omega_{f}$ and  $(\rho^r,q^r)\in \Omega_{c}$,
      left, the first case in the proof of Proposition~\ref{prop:2},
      right the second case in the proof of Proposition~\ref{prop:2}.}
    \label{fig:comparison_RP}
  \end{figure} 
\end{proof}

%
%
\Section{The Godunov Method}
\label{sec:Godunov}

In this section we describe the Godunov method for the
model~(\ref{eq:Modeleta}) under the assumptions
\textbf{(H-1)}, \textbf{(H-2)}, and \textbf{(H-3)};
see~\cite{MR0119433, LeVeque, MR1925043}.
We recall that such numerical scheme is a finite volume method,
based on Riemann problems
within computational cells in order to obtain the numerical fluxes.

Consider the system~(\ref{eq:Modeleta}), with the conserved variables
$(\rho, \eta)$; see Figure~\ref{fig:phases}.
Introduce the space discretization $\Delta x$ and the time dicretization
$\Delta t$ satisfying the \textit{Courant-Friedrichs-Lewy}
conditions (see~\cite{MR1512478, MR0213764}) as in~\cite{LeVeque}.
Define $\nu=\Delta t/\Delta x$, and,
for all $j \in \interi$ and all $n \in \naturali$,
the points $x_{j+1/2} = j \Delta x$, $x_{j} = \left(j - 1/2\right) \Delta x$,
the time $t^{n} = n \Delta t$, and
the cell $C_{j}^{n} = \left\lbrace t^{n}\right\rbrace \times 
\left[ x_{j-1/2}, x_{j+1/2}\right]$ of length $\Delta x$.

The Godunov method consists in the following steps.
\begin{enumerate}
\item Approximate an initial condition $\left(\rho_0, \eta_0\right)$
  by a function
  $\left(\rho_\nu(0, \cdot), \eta_\nu(0, \cdot)\right)$, constant in each
  cell $C_j^0$. Define $u_j^0 = \left(\rho_j^0, \eta_j^0\right)
  = \left(\rho_\nu(0, x_j), \eta_\nu(0, x_j)\right)$ for every
  $j \in \interi$.

\item Construct, for every $n \in \naturali$ and $j \in \interi$,
  the values $u^{n+1}_j$ by the recurrence formula
  \begin{equation}
    \label{eq:NumRec}
    \begin{split}
      u^{n+1}_{j} & = \left(\rho_j^{n+1}, \eta_j^{n+1}\right) =
      u_{j}^{n} - \frac{\Delta t}{\Delta x}
      \left(\mathcal{F}^{n}_{j+1/2} -
        \mathcal{F}^{n}_{j-1/2}\right)
      \\
      & = \left(\rho_j^{n}, \eta_j^{n}\right)
      \! - \! \frac{\Delta t}{\Delta x}
      \left[\left({F}^{n}_{j+1/2}, {G}^{n}_{j+1/2}\right) \!-\!
        \left({F}^{n}_{j-1/2}, {G}^{n}_{j-1/2}\right)\right],
    \end{split}
  \end{equation}
  where $\mathcal{F}^{n}_{j+1/2} = \left({F}^{n}_{j+1/2}, {G}^{n}_{j+1/2}\right)$
  is a numerical flux at the interface
  $x_{j+1/2}$, obtained through the solution of corresponding Riemann problems.
\end{enumerate}
\begin{figure}[t!]
  \centering
  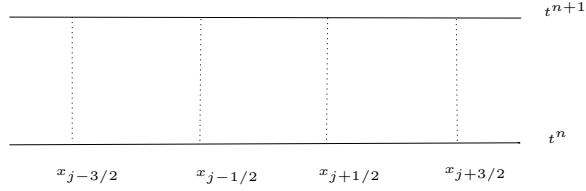

  \caption{A cell given by $\left[x_{j-1/2}, x_{j+1/2}\right]$ in the Godunov
    Scheme.}
  \label{fig:cella}
\end{figure}
If $j \in \interi$, the numerical flux
$\mathcal{F}^{n}_{j+1/2} = \left(F^{n}_{j+1/2}, G^{n}_{j+1/2} \right)$
depends on the values $u^n_j$ and $u^n_{j+1}$. Namely, denoting for simplicity
with $u_l = \left(\rho_l, \eta_l\right)$ and $u_r = \left(\rho_r, \eta_r\right)$
respectively $u^n_j$ and $u^n_{j+1}$, we have
\begin{equation}
  \label{eq:FluFirst}
  F^{n}_{j+1/2} =\left\{
    \begin{array}{l@{\qquad}l}
      \rho_{l}\vmax  
      & \mbox{ if } \quad u_{l},u_{r} \in F
      \\
      \eta_{m}\psi(\rho_{m}) 
      & \mbox{ if } \quad u_{l},u_{r} \in C
      \\
      \rho_{\circ } \vmax  
      &  \mbox{ if } \quad u_{l}\in C, u_{r} \in F
      \\
      \eta_{m}\psi(\rho_{m}) 
      &  \mbox{ if } \quad u_{l}\in F,\, u_{r} \in C,\, \rho_{l}\vmax>\eta_{m}\psi(\rho_{m})
      \\
      \rho_{l}\vmax 
      & \mbox{ if } \quad u_{l}\in F,\, u_{r} \in C,\,  \rho_{l}\vmax<\eta_{m}\psi(\rho_{m})
       \end{array}
  \right.
\end{equation}
and
\begin{equation}
  \label{eq:FluSecond}
  G^{n}_{j+1/2} =\left\{
    \begin{array}{ll}
     \eta_{l}\vmax  
      & \mbox{ if } \quad u_{l},u_{r} \in F
      \\
      \frac{\eta_{m}^{2}}{\rho_{m}}\psi(\rho_{m}) 
      & \mbox{ if } \quad u_{l},u_{r} \in C
      \\
      \eta_{\circ }\vmax 
      & \mbox{ if } \quad u_{l}\in C,\, u_{r} \in F
      \\
      \frac{\eta_{m}^{2}}{\rho_{m}}\psi(\rho_{m}) 
      & \mbox{ if } \quad u_{l}\in F,\, u_{r} \in C,\,
        \rho_{l}\vmax>\eta_{m}\psi(\rho_{m})
      \\
      \eta_{l}\vmax 
      & \mbox{ if } \quad u_{l}\in F,\, u_{r} \in C,\, 
        \rho_{l}\vmax<\eta_{m}\psi(\rho_{m})\,.
       \end{array}
  \right.
\end{equation}
Above $(\rho_{m},\eta_{m}) \in C$ is the middle state in the solution to
the Riemann problem~(\ref{eq:RD}); see Subsection~\ref{sse:RP}.
Moreover the point $(\rho_{\circ},\eta_{\circ}) \in F \cap C$ is the middle point
in the solution to the Riemann problem~(\ref{eq:RD}) when
$u_l \in C$ and $u_r \in F$; see point (3) of Subsection~\ref{sse:RP}.
Note that the condition $\rho_{l}\vmax$ respectively greater or less than $\eta_{m}\psi(\rho_{m})$ means that, by the Rankine-Hugoniot condition, the phase transition wave connecting $(\rho_{l},\eta_{l})\in F$ and $(\rho_{m},\eta_{m})\in C$ has strictly negative or positive speed.

\subsection{Example: the case of a traffic light with increasing $w$}
\label{sse:ex1}
We present here a simple situation of a road with a traffic light, which turns
green at the initial time $t=0$.
At the beginning all the vehicles are stopped at the maximal density
and $w$ is linearly increasing between $\check w$ and $\hat w$.
We assume that the road is $3$Km long, modeled by the real 
interval $(0, 3000)$ and the traffic light is located at position $x = 500$.

We choose,  for the model~(\ref{eq:Modeleta}), the parameters
\begin{equation}
  \label{eq:parameters-es1}
  R = 1, \quad \vmax = 60\, \textrm{Km/h},
  \quad \check w=120\, \textrm{Km/h}, \quad \hat w = 140\, \textrm{Km/h}, 
\end{equation}
with $\psi(\rho) = 1 - \rho$ and initial conditions
\begin{equation}
  \label{eq:init-cond-es1}
  \rho_0(x) = \left\{
    \begin{array}{ll}
      1, 
      & x \le 500
      \\
      0,
      &
        x > 500
    \end{array}
  \right.
  \quad
  w_0(x) = \left\{
    \begin{array}{ll}
      \check w + \frac{\hat w - \check w}{500} x, 
      & x \le 500
      \\
      0,
      &
        x > 500.
    \end{array}
  \right.
\end{equation}
In Figure~\ref{fig:contour-es-1} we show the contour plot for the
numerical solution, obtained with the Godunov scheme, using the
following numerical parameters:
\begin{equation}
  \label{eq:num-par-es1}
  T = 300\, \textrm{s}, \quad \Delta x = 1\, \textrm{m}, \quad
  \Delta t = 0.7 \, \frac{\Delta x}{\vmax}.
\end{equation}

\begin{figure}
  \centering
  \includegraphics[width=0.45\linewidth]{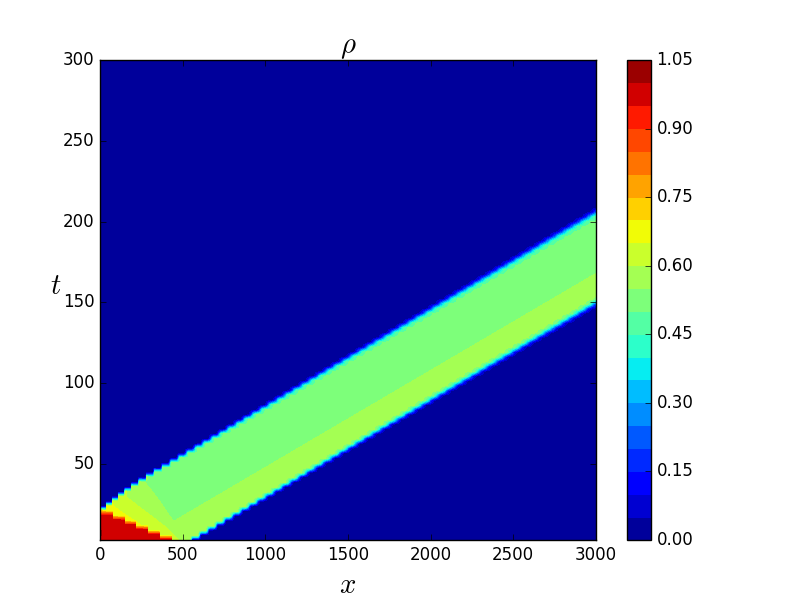}
  \includegraphics[width=0.45\linewidth]{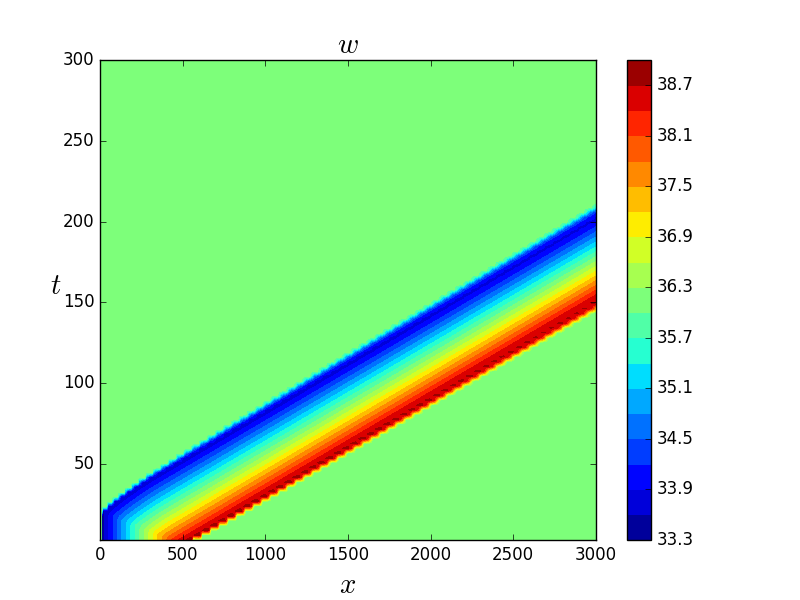}
  \caption{The contour plot for the density $\rho$ and for the 
  maximal speed $w$ obtained with the Godunov method, with
  the parameters~(\ref{eq:parameters-es1}) and~(\ref{eq:num-par-es1}),
  and with $\psi(\rho) = 1 - \rho$.
  The initial condition are in~(\ref{eq:init-cond-es1}).}
  \label{fig:contour-es-1}
\end{figure}

\subsection{Example: the case of a traffic light with decreasing $w$}
\label{sse:ex2}
We consider here a situation very similar to that of Subsection~\ref{sse:ex1}.
The only difference is that the initial $w$ is decreasing
between $\check w$ and $\hat w$.
As before, we assume that the road is $3$Km long, modeled by the real 
interval $(0, 3000)$ and the traffic light is located at position $x = 500$.

We choose,  for the model~(\ref{eq:Modeleta}),
the parameters~(\ref{eq:parameters-es1}), the function
$\psi(\rho) = 1 - \rho$, initial conditions
\begin{equation}
  \label{eq:init-cond-es2}
  \rho_0(x) = \left\{
    \begin{array}{ll}
      1, 
      & x \le 500
      \\
      0,
      &
        x > 500
    \end{array}
  \right.
  \quad
  w_0(x) = \left\{
    \begin{array}{ll}
      \hat w + \frac{\check w - \hat w}{500} x, 
      & x \le 500
      \\
      0,
      &
        x > 500.
    \end{array}
  \right.
\end{equation}
In Figure~\ref{fig:contour-es-2} we show the contour plot for the
numerical solution, obtained with the Godunov scheme, using the
numerical parameters~(\ref{eq:num-par-es1}).

\begin{figure}
  \centering
  \includegraphics[width=0.45\linewidth]{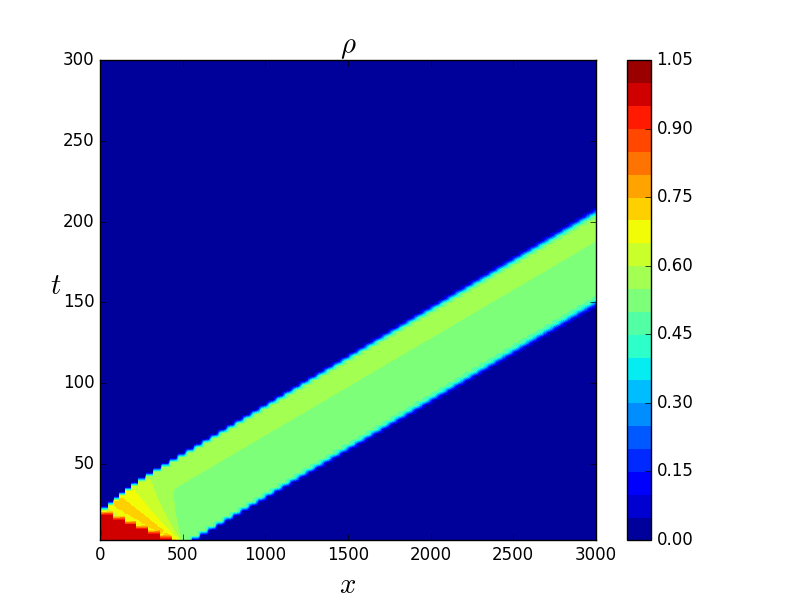}
  \includegraphics[width=0.45\linewidth]{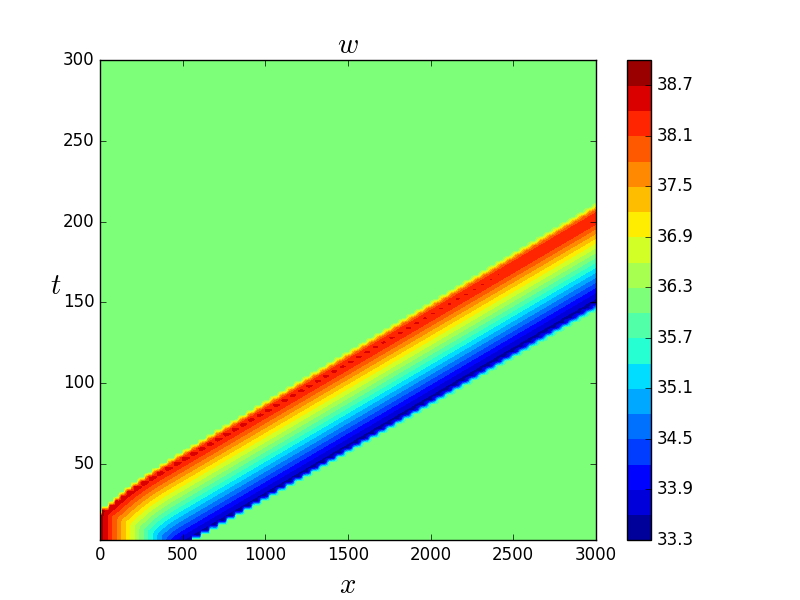}
  \caption{The contour plot for the density $\rho$ and for the 
  maximal speed $w$ obtained with the Godunov method, with
  the parameters~(\ref{eq:parameters-es1}) and~(\ref{eq:num-par-es1}),
  and with $\psi(\rho) = 1 - \rho$.
  The initial condition are in~(\ref{eq:init-cond-es2}).}
  \label{fig:contour-es-2}
\end{figure}

\subsection{Example: increasing $\rho$ and decreasing $w$}
\label{sse:ex3}
In this part we consider an example where the initial conditions are
monotone functions: the density $\rho$ is increasing, while the
maximal speed $w$ is decreasing.
We assume that the road is $3$Km long, modeled by the real 
interval $(0, 3000)$.

We choose the parameters~(\ref{eq:parameters-es1}), the function
$\psi(\rho) = 1 - \rho$, and the initial conditions
\begin{equation}
  \label{eq:init-cond-es3}
  \begin{split}
    \rho_0(x) & = \left\{
      \begin{array}{ll}
        0, 
        & x \le 500
        \\
        0.2 + 0.5\, \frac{x -500}{2000},
        & 500 < x < 2500
        \\
        0,
        &
          x > 2500
      \end{array}
    \right. 
    \\
    w_0(x) & = \left\{
      \begin{array}{ll}
        0,
        &
          x \le 500
        \\
        \hat w + \frac{\check w - \hat w}{2000} \left(x - 500\right), 
        & 500 < x < 2500
        \\
        0,
        &
          x \ge 2500.
      \end{array}
    \right.
  \end{split}
\end{equation}
In Figure~\ref{fig:contour-es-3} we show the contour plot for the
numerical solution, obtained with the Godunov scheme, using the
following numerical parameters:
\begin{equation}
  \label{eq:num-par-es3}
  T = 200\, \textrm{s}, \quad \Delta x = 1\, \textrm{m}, \quad
  \Delta t = 0.7 \, \frac{\Delta x}{\vmax}.
\end{equation}

\begin{figure}
  \centering
  \includegraphics[width=0.45\linewidth]{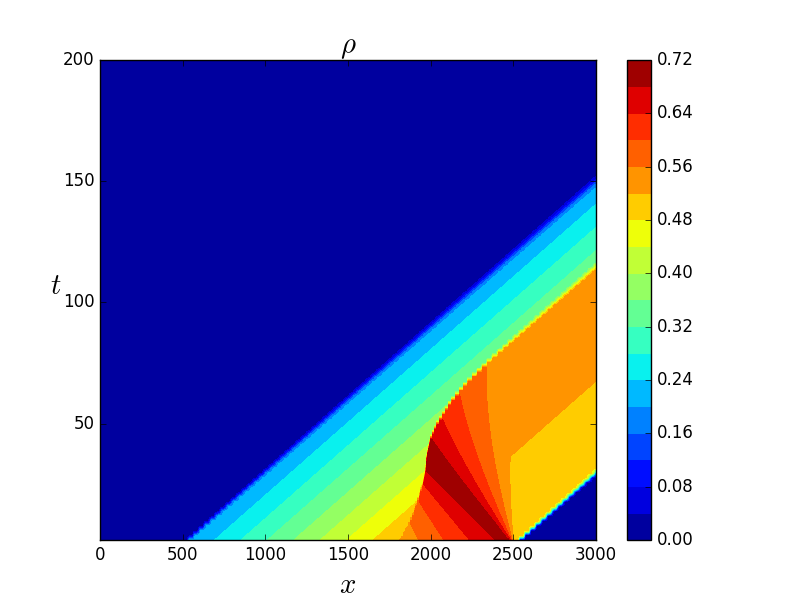}
  \includegraphics[width=0.45\linewidth]{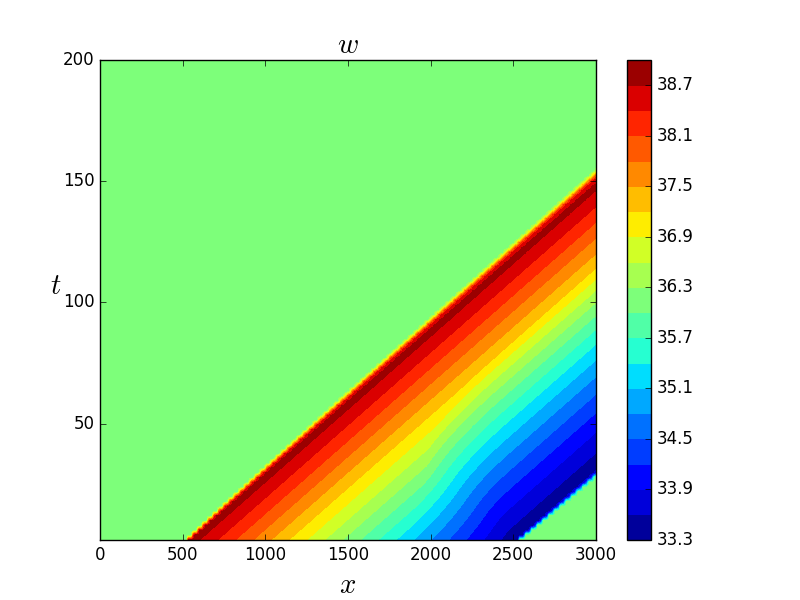}
  \caption{The contour plot for the density $\rho$ and for the 
  maximal speed $w$ obtained with the Godunov method, with
  the parameters~(\ref{eq:parameters-es1}) and~(\ref{eq:num-par-es3}),
  and with $\psi(\rho) = 1 - \rho$.
  The initial condition are in~(\ref{eq:init-cond-es3}).
  A shock curve is generated at about $x \sim 1800$ and interacts with
  a rarefaction curve generated at about $x \sim 2500$.}
  \label{fig:contour-es-3}
\end{figure}

\section*{Acknowledgments}
The authors were partial supported by the INdAM-GNAMPA 2016 project ``Balance Laws: Theory and Applications''.

{\small{

    \bibliographystyle{abbrv}

    \bibliography{model} }}
\end{document}